\documentclass[11pt,reqno]{article}
\usepackage{amssymb,amsmath,amsthm}
\usepackage[utf8]{inputenc}
\usepackage{tikz}
\usepackage{hyperref}
\usepackage{xcolor}
\usepackage{authblk} 

\usepackage{ulem}
\usepackage[shortlabels]{enumitem}
\setlist[enumerate]{topsep=0pt,itemsep=-1ex,partopsep=1ex,parsep=1ex}
\usepackage{graphicx}

\newtheorem{statement}{}[section]
\newtheorem{theorem}[statement]{Theorem}
\newtheorem{lemma}[statement]{Lemma}
\newtheorem{proposition}[statement]{Proposition}
\newtheorem{definition}[statement]{Definition}
\newtheorem{corollary}[statement]{Corollary}
\newtheorem{remark}[statement]{Remark}

\makeatletter
\newcommand{\subjclass}[2][1991]{%
  \let\@oldtitle\@title%
  \gdef\@title{\@oldtitle\footnotetext{#1 \textbf{Mathematics subject classification:} #2}}%
}
\newcommand{\keywords}[1]{%
  \let\@@oldtitle\@title%
  \gdef\@title{\@@oldtitle\footnotetext{\textbf{Key words and phrases:} #1}}%
}
\makeatother

\def\n{|\cdot|}
\def\nn{{\|\cdot\|}}
\newcommand{\tnorm}{|\hskip-.07em\|}
\def\nnn{\tnorm\cdot\tnorm}
\newcommand{\ntres}[1]{\tnorm #1 \tnorm} 
\newcommand{\ndos}[1]{\|#1 \|} 

\def\NN{\mathbb N}
\def\IN{\hbox{{\rm I}\kern-.13em{\rm N}}}

\def\RR{\mathbb{R}}
\def\IR{\hbox{{\rm I}\kern-.13em{\rm R}}}

\def\conv{{\rm conv}\,}

\def\sp{\hbox{{\rm span}}}

\def\aa{\alpha}

\def\la{\langle}
\def\ra{\rangle}

\title{Octahedrality and G\^ateaux smoothness}

\subjclass[\textbf{2020}]{46B03, 
46B20, 
46B26. 
}

\author{Ch. Cobollo\footnote{
Christian Cobollo. Instituto Universitario de
Matem\'atica Pura y Aplicada. Universitat Polit\`ecnica de Val\`encia
(Spain).
Email: chcogo@upv.es. Corresponding author.
} \ and
P. H\'ajek\footnote{
Petr H\'ajek. Department of Mathematics, Faculty of Electrical Engineering.
Czech Technical University in Prague (Czech Republic).
Email: hajekpe8@fel.cvut.cz
}
}

\date{July 2024}

\begin{document}
\maketitle
\begin{abstract}
We prove that every Banach
space admitting a G\^ateaux smooth norm and containing a complemented copy of
$\ell_1$ has an equivalent renorming which is simultaneously G\^ateaux smooth and octahedral.  
This is a partial solution to a problem from the early nineties.
\end{abstract}

\section{Introduction} \label{sec:O+G-intro}

Octahedral norms were introduced by Godefroy at \cite{G1}---see also in \cite{GM}---where it was proved that a Banach space $X$
admits an octahedral norm if and only if
it contains a copy of $\ell_1$. Octahedrality can be viewed as a
strong non-differentiability---in the sense of Fr\'echet---
condition. Octahedrality has been extensively
studied in recent years, see e.g. \cite{A,BLZ,D,LL,PZ} and many others,
 in connection with the various
dentability conditions on a Banach space. It also has interesting applications in connection with the
weak sequential completeness of the Banach space, see \cite{DGZ,G1}, and it is also connected with the problem of preserved G\^ateaux
smoothness points---see \cite{G1,WK}.

The coexistence of G\^ateaux smoothness and octahedrality
for a single norm has, therefore, been a known problem
since the early nineties---see \cite{WK,Z} and \cite[Problem 7]{HMZ12}).
Curiously, up to date, the only example of a Banach space $X$ which
admits such a norm is the Hardy space $H_1(D)$, as shown
(using deep results from harmonic analysis) in the monograph
\cite[p. 120]{DGZ}. This space is a separable
subspace of $L_1$ containing a complemented copy of
$\ell_1$. To find such a renorming, even in the basic case $X=\ell_1$, was open.
Our main result in this note is the following.

\begin{theorem}\label{thm:oct}
Let $(X,\nn)$ be a Banach space admitting a G\^ateaux smooth
(equivalent) norm and having a complemented subspace isomorphic to $\ell_1$.
Then $X$ admits a renorming $\ntres{\cdot}$ which is simultaneously
octahedral and G\^ateaux smooth.
\end{theorem}

By \cite[Lemma 28]{HMZ12}, a given element $x\in S_X$ is a very smooth point if
and only if
$x$ is a point of G\^ateaux smoothness in the bidual $X^{**}$.
By \cite{WK}, an octahedral G\^ateaux smooth norm has no points of
preserved smoothness, so our main result implies, in particular, that
under the assumptions of Theorem \ref{thm:oct}
$X$ has a G\^ateaux smooth norm without points of preserved smoothness,
i.e. no very smooth points. Notice that the notion of very smoothness coincides with the one of strong G\^ateaux smoothness---see \cite{HMZ12}.


\vskip5mm

Our proof relies on a new method of construction
based on controlled directional estimates of the norm
on a dense subspace, which
passes to the completion. It is somewhat subtle, and
uses the complementability of $\ell_1$ heavily.
To some extent, this is inevitable, as octahedral norms
cannot have a rotund dual norm, which is the standard
condition in order to obtain a G\^ateaux smooth norm.
Indeed, octahedral spaces contain an asymptotically isometric
$\ell_1$-sequence (\cite{ALNT}), and spaces with such a 
sequence cannot have a dual rotund norm (\cite{NPT}).
The proof would be no simpler if we just assumed that $X=\ell_1$,
but it is not clear if there is a simple formal argument
in our case, using the special case of $\ell_1$, together
with the complementability of $\ell_1$ in $X$.

\medskip 

We are inclined to believe that the complementability condition
in Theorem \ref{thm:oct} is redundant (so, the containment of $\ell_1$
should be sufficient and, of course, also necessary), but
our method of proof does not cover this case.

\medskip

The rest of the text is devoted to the proof of the main Theorem \ref{thm:oct} through the construction of a renorming $\ntres{\cdot}$ being simultaneously G\^ateaux smooth and octahedral. The document is organized as follows: the remaining part of this introductory section will contain preliminaries and notation. Section \ref{sec:O+G-const} consist of the inductive construction of the renorming and the proof of its elementary properties. Lemma \ref{lem:O+G-equiv} shows that it is an equivalent norm to the original one, and Proposition \ref{ocet} contains the argument for octahedrality. The last Section \ref{sec:O+G-smoothness} is completely dedicated to showing the G\^ateaux smoothness of the final norm, which is the most delicate part of the proof. It consists of showing that the G\^ateaux smoothness on the original construction---Proposition \ref{prop:O+G-diff-Xn} and Corollary \ref{cor:O+Gdiff-Y}---is inherited to the whole space $X$. The argument depends on some suitable estimates of the directional derivatives and splitting in two cases, depending if there exists a Birkoff--James orthogonal relation between the point and the direction or not---Subsections \ref{subsec:O+G-tangent} and \ref{subsec:O+G-not-tangent}, respectively.

\subsection{Preliminaries and notation}

We assume that our Banach space $(X,\nn)$
has a G\^ateaux smooth norm $\nn$, and $X\cong X_0\oplus\ell_1$, for some
Banach space
$X_0$. We will use $\{e_i\}_{i=1}^\infty$ to denote the canonical basis of $\ell_1$.
For every $n\in \NN$, consider the linear subspace
\[
X_n:=X_0+\sp\{e_i:\; 1\le i\le n\},\]
and put 
\[\displaystyle Y:=\bigcup_{n\in\mathbb{N}}X_n=\{x\in X: x\in X_n \text{ for some }n\in \NN\}.\]
 Thus, the whole space $X$ is the completion of the subspace $Y$.

It is clear that there exists a unique decomposition of every $x\in X$ as
\[
x=x_0+\sum_{j=1}^\infty x_j e_j,\;\; \text{ where }x_0\in X_0,\;\;
\sum_{j=1}^\infty|x_j|<\infty.
\]
For any $n\in \NN$,  consider the $n$-th canonical projection (or the canonical projection to $X_n$) as the map $P_n:X\to X_n$, 
\[
 P_n(x):=x_0+\sum_{j=1}^n x_j e_j.
\]
 To simplify the notation, we will also denote the $n$-th canonical projection of a given element $x$ by the symbol
\[
x^n:=P_n(x).
\]

The final norm $\nnn$ will be obtained 
through the construction of a sequence of compatible renormings
$\nnn_n$ on the spaces $X_n$. Such a sequence has, of course,
a unique extension to the whole space $X$. It will be easy to check that $\nnn$ is octahedral, as it will have the property on $Y$, and octahedrality
passes to the completion $X$. The construction will also be
G\^ateaux smooth at all points of $Y$---see Section \ref{sec:O+G-smoothness}. The difficult part of the argument is to prove
the G\^ateaux smoothness for every $x\in X\setminus Y$. This is
equivalent to the existence of all directional derivatives
$\frac{\partial\ntres{x}}{\partial h}$, where $h\in Y$.

We refer to \cite{DGZ,yellow2,GMZ2,HJ} for standard results and notation.

\section{The construction of the norm} \label{sec:O+G-const}
This section will contain the inductive construction of the norm $\nnn$, which will be octahedral and G\^ateaux smooth. As said before, it will start through an inductive process of constructing norms in the spaces $X_n$. Roughly speaking, the main idea behind this is to add one more dimension and construct a new norm as the Minkowski functional of a new convex body, defined through homothetic copies of the previous unit ball. Through the assumptions of the function that indicates the homothetic factor depending on the height, we may achieve the new convex body is still smooth, and that on each step, the norm for the new vectors is ``asymptotically'' an $\ell_1$-sum.

We will start by consider sequences of real numbers $\{z_n\}_{n\in \NN}$, $\{l_n\}_{n\in \NN}$, and $\{s_n\}_{n\in \NN}$ such that:

\begin{itemize}
    \item $0<z_n<l_n<s_n<1$;
    \item $z_n$ strictly decreasing;
    \item $l_n$ strictly decreasing, $\{l_n\}_{n \in \NN}\in \ell_1$;
    \item $s_n$ strictly increasing, $s_n \to 1$.
\end{itemize}

Also, take a sequence of continuous convex and  real-valued functions $\{f_n\}_{n\in \NN}$, $f_n\colon [0,1]\to \RR$  with the following properties,

\begin{enumerate}
    \item $f_n\equiv 0$ in $[0,z_n]$;
    \item $f_n$ is smooth and strictly increasing in $[z_n,l_n]$;
    \item $f_n(t):= |t|-\dfrac{z_n+l_n}{2}$;
    \item $f_n$ in
$[z_n,1]$ is strictly
increasing, smooth,
$f(1)=1$ and ${\lim_{t\to 1} f'(t)=\infty}$.
\end{enumerate}

Then, the following holds.
\begin{proposition}
  By the construction above,  for every $t\in [0,1]$
  \[t-\dfrac{z_n+l_n}{2} \leq f_n(t) \leq t.\]

  In particular, $f_n\to |\cdot |$ uniformly.
\end{proposition}
\begin{proof}
    Is is clear from the construction. See also Figure \ref{fig:n-func}.
\end{proof}

\begin{figure}[h]
\begin{center}
\includegraphics[width=4in]{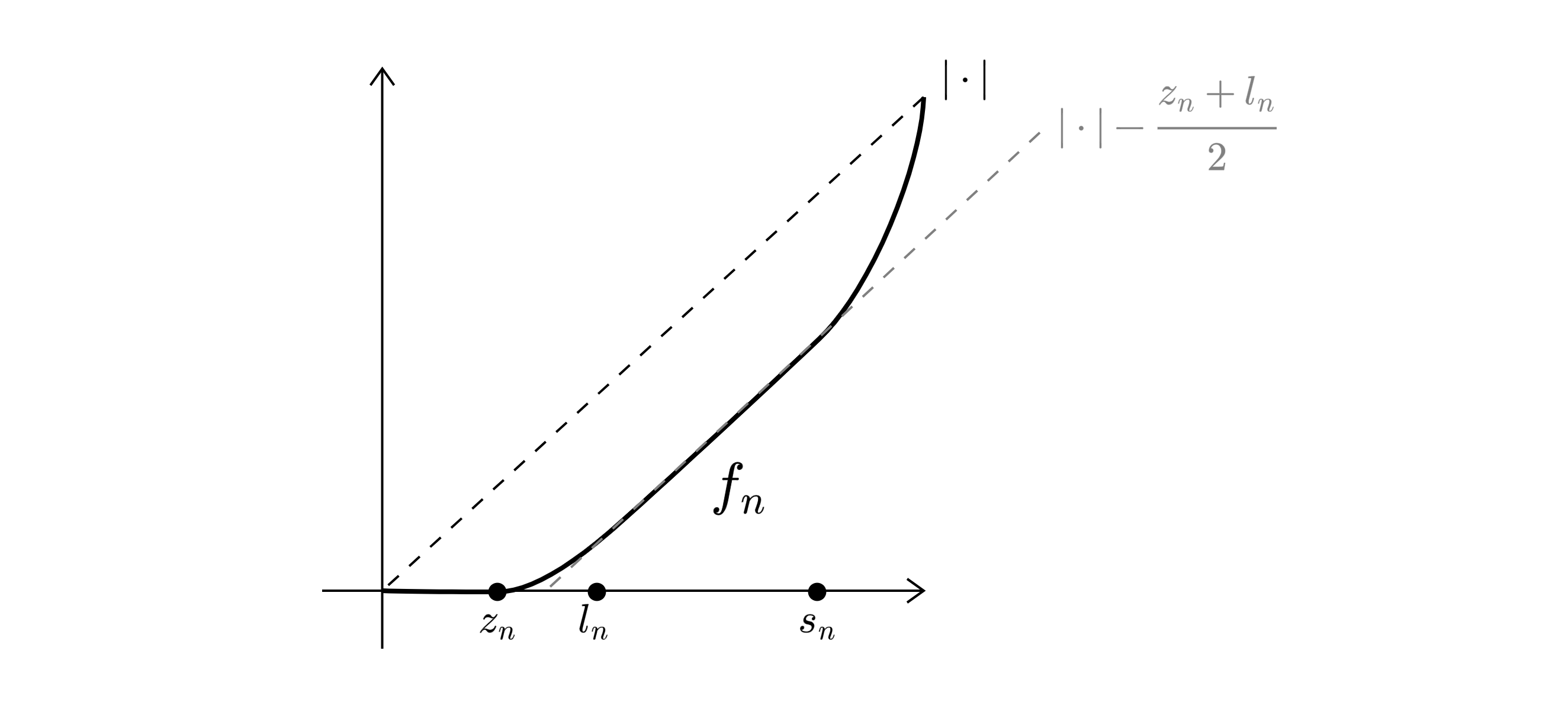}
\caption{Shape of the function $f_n$}
\label{fig:n-func}
\end{center}
\end{figure}

From the result above we also have that for every $t\in [0,1],$

 \[1-t+\dfrac{z_n+l_n}{2} \geq 1- f_n(t) \geq 1-t.\]

Now, we are ready to start with the construction of the norms. For $n=0$ just define $\nnn_0:=\|\cdot\|$ as the restriction of the
G\^ateaux smooth
norm from $X$ to $X_0$.
For $n\geq 1$, we will define a (equivalent) norm $\nnn_n$ in $X_n$ by
the Minkowski functional of the set
\begin{equation}\label{eq:n-ind}
B_n:=\{x\in X_n : \ntres{P_{n-1}x}_{n-1}
\leq 1-f_n(|x_n|), x_n\in [-1,1]\}.
\end{equation}

Thus, $\nnn_n:= \mu_{B_n} $, and so $B_{X_n}=B_n$ and 
\begin{equation}\label{eq:gat-oct-sphere}
    S_{X_n}=\{x\in X_n: \ntres{P_{n-1}x}_{n-1}=1-f_n(|x_n|), x_n\in[-1,1] \} .
\end{equation}

\begin{figure}[h]
\begin{center}
\includegraphics[width=4in]{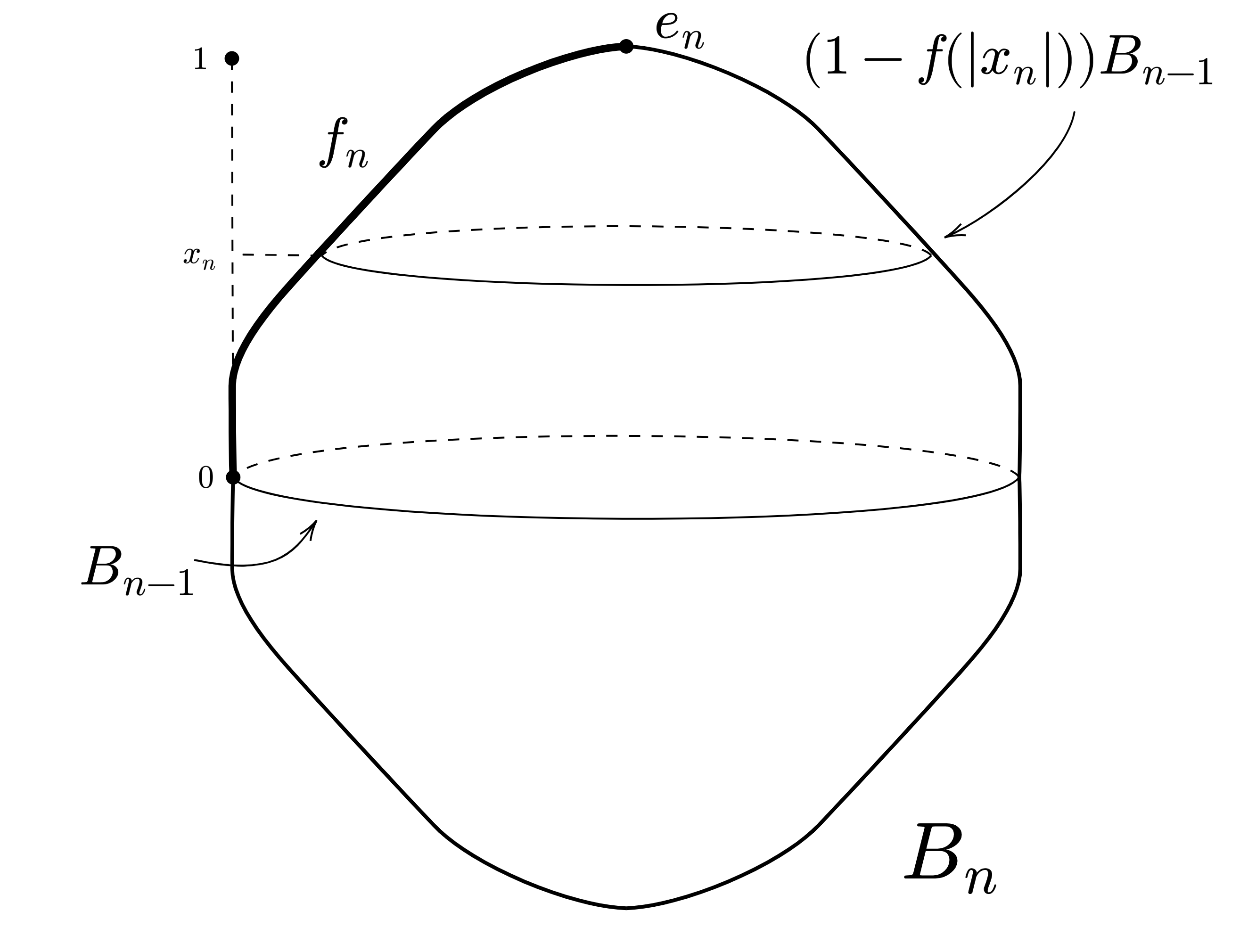}
\caption{Shape of the unit ball $B_n$. The slice of $B_n$ at height $x_n$ is a homotetic copy of $B_{n-1}$ with scalar factor $1-f_n(x_n)$.}
\label{fig:n-ball}
\end{center}
\end{figure}

\begin{lemma}\label{lem:nnn-homotecy}
    Let $x\in Y$. Then, 
    \[\ntres{P_{n-1}x}_{n-1}=(1-f_n(\frac{|x_n|}{\ntres{P_n x}_n})) \ntres{P_n x}_n.\]
    
    In particular, $\ntres{P_{n-1}x}_{n-1}\leq \ntres{P_{n}x}_{n} $.
\end{lemma}
\begin{proof}
   We can assume without loss of generality that $\ntres{P_n x}_n>0$---otherwise, the result is trivial. Then, $P_n \Big( \dfrac{x}{\ntres{P_n x}_n} \Big) \in S_{X_n}$, or equivalently,

   \[\ntres{P_{n-1}\Big(P_n \Big( \dfrac{x}{\ntres{P_n x}_n} \Big) \Big)}_{n-1} = 1-f_n(\dfrac{|x_n|}{\ntres{P_n x}_n}),\]

   and from here we deduce

   \[\ntres{P_{n-1} x}_{n-1}= (1-f_n(\dfrac{|x_n|}{\ntres{P_n x}_n})) \ntres{P_n x}_n \leq \ntres{P_n x}_n. \]
\end{proof}

For the previous description of the norm there are some easy consequences that follow naturally. We will state them for further reference.

\begin{corollary}\label{cor:nnn-homotecy}

If $x\in S_{(X_{n},\ntres{\cdot}_{n})}$ and $|t|<z_{n+1}$, then 
\[
\ntres{x+t e_{n+1}}_{n+1}=\ntres{x}_{n}.
\]

In particular,
 \[
X_{n}\cap B_{(X_{n+1},\ntres{\cdot}_{n+1})}= B_{(X_{n},\ntres{\cdot}_{n})}.
\]
\end{corollary}
\begin{proof}
    The formulae follow readily from the
properties of $f_n$ and the construction of $\ntres{\cdot}_n$.
\end{proof}

Now, we are ready to define the final renorming $\nnn$, through the supremum of the already constructed $\nnn_n$, that is, for any $x\in X$
\[{\ntres{x}:= \sup_{n\in \NN}\{\ntres{P_nx }_n\}}.\]

We will prove that it is indeed an equivalent norm though being equivalent to the already equivalent norm $\nnn_0\oplus_1 \nn_1$ in $X$ (the computation of $\ntres{x_0}_0+\ndos{(x_n)_{n=1}^\infty}_1$ for any $x\in X$). For this purpose, we may define some other norms to use as a comparison.

\medskip

Proceeding by induction again, for $X_0$ take $\n_0:= \nnn_0$---the original G\^ateaux norm in $X_0$. For $n\geq 1$, consider $\n_n:= \nnn_{n-1}\oplus_1 \n$. It is straightforward to see that for $n\geq 1$ the unit ball associated with this norm is
\[B_{\n_n}:= \{x\in X_n: \ntres{P_{n-1}x}_{n-1}\leq 1- |x_n|\}= \conv(B_{n-1}, e_n).\]

\begin{remark}
    In the following, we will use the fact that $\prod_{n\in\NN} (1+\frac{z_n+l_n}{2})$ converges. This is due to the fact that $\{\frac{z_n+l_n}{2}\}_{n\in\NN}\in \ell_1$, because in the construction of the functions $f_n$ we took $\{l_n\}_{n\in\NN}\in \ell_1$.
\end{remark}
\begin{lemma} \label{lem:O+G-equiv}
   For the constructed norms, it is satisfied that
    \[\dfrac{1}{1+\frac{z_n+l_n}{2}} \n_n \leq \nnn_n \leq \n_n.\]

    In particular, by considering the norm $\nnn_0\oplus_1 \nn_1$ on $X$, we have that
    \[ \dfrac{1}{\prod_{n\in\NN} (1+\frac{z_n+l_n}{2})} (\nnn_0\oplus_1 \nn_1)\leq \nnn \leq\nnn_0\oplus_1 \nn_1.\]
\end{lemma}
\begin{proof}
First, it is clear that
    \[B_{\n_n}\subset B_n \subset  (1+\frac{z_n+l_n}{2}) B_{\n_n}.\]

    Now, on the one hand, from the right-hand side inclussion, we have that for every $n$,
    \begin{align*}
       \ntres{P_n x}_n \leq |P_nx|_n= \ntres{P_{n-1}x}_{n-1}+ |x_n|.
    \end{align*}
    Applying this inequation iteratively, we reach
    \[\ntres{P_n x}_n \leq \ntres{x_0}_0+\sum_{i=1}^n |x_i|.\]

    In particular, by taking supremums at both sides, we reach the first inequality ${\nnn \leq \nnn_0\oplus_1 \nn_1}$.

On the other hand, by the left-hand side inclussion, we have
    \begin{align*}
       \ntres{P_n x}_n \geq \dfrac{1}{1+\frac{z_n+l_n}{2}} \|P_nx\|_n= \dfrac{1}{1+\frac{z_n+l_n}{2}} \Big( \ntres{P_{n-1}x}_{n-1}+ |x_n|  \Big).
    \end{align*}

    Once again, applying this iteratively
     \begin{align*}
        \ntres{P_n x}_n 
 \geq   \Big( \prod_{j=1}^{\infty}  \dfrac{1}{1+\frac{z_j+l_j}{2}}\Big)  \Big(  \ntres{x_0}_0+  \sum_{i=1}^n |x_i|\Big). 
    \end{align*} 

    In particular, by taking supremums on $n\in \NN$, we reach
    \[\nnn \geq  \dfrac{1}{\prod_{n\in\NN} (1+\frac{z_n+l_n}{2})} (\nnn_0\oplus_1 \nn_1).\]
    The proof is over.
\end{proof}

\subsection{Octahedrality of the norm}
Here, we will show the octahedrality of $\nnn$.
Recall that a norm $\nn$ of a Banach space $X$ is said to be \textbf{octahedral} if for every $\varepsilon>0$ and every finite-dimensional subspace $F$ of $X$ there exists $x\in S_X$ such that
\[\|y+\aa x\| \geq (1-\varepsilon)(\|y\|+|\aa|)\]
for every $y\in F$ and $\aa\in \RR$.

 The core of the idea is that, in the set $B_n$, the element $e_n$ is ``close'' to witnessing the octahedrality of the norm for any element in $B_{n-1}$, and the closeness is reduced with the increasing of the $n$.
 
\begin{proposition}\label{ocet}
    Let $\varepsilon>0$. Then, there exists $n_0\in \NN$ such that, for any $n\geq n_0$,
    \[\ntres{P_{n-1} x+ \aa e_n}_n \geq (1-\varepsilon) (\ntres{P_{n-1} x}_n + |\aa|).  \]

    In particular, the norm $\ntres{\cdot}$ is octahedral.
\end{proposition}
\begin{proof}
    Take $n_0\in \NN $ such that $\frac{z_n+l_n}{z_n+l_n+2} \leq \varepsilon$. Then,
    \begin{align*}
        \ntres{P_{n-1}x+\aa e_n}_n &\geq \dfrac{1}{1+\frac{z_n+l_n}{2}} \|P_{n-1}x+\aa e_n\|_n \\
        &= \dfrac{1}{1+\frac{z_n+l_n}{2}} \Big( \ntres{P_{n-1}x}_{n-1}+|\aa| \Big)\\
        &= \dfrac{1}{1+\frac{z_n+l_n}{2}} \Big( \ntres{P_{n-1}x}_{n}+|\aa| \Big),
    \end{align*}
    and this last term is greater or equal than $(1-\varepsilon)(\ntres{P_{n-1}x}_{n}+|\aa|)$ if and only if
    \[1-\dfrac{1}{1+\frac{z_n+l_n}{2}}= \frac{z_n+l_n}{z_n+l_n+2} \leq \varepsilon.\]

    As this is satisfied because of the choice of $n\in\NN$, we conclude the proof.
\end{proof}

\section{G\^ateaux smoothness of the norm}
\label{sec:O+G-smoothness}

This final section is fully devoted to showing the G\^ateaux smoothness of $\nnn$. It will require splitting the argument into several steps. First, the smoothness of the norm $\nnn_n$ in $X_n$. Geometrically, this is due to the properties on the functions $f_n$---see again Figure \ref{fig:n-ball}. The smoothness of $\nnn_n$ in the points that belong to the previous $X_{n-1}\backslash\{0\}$ is due to the assumption $f_n\equiv 0$ in $[0,z_n]$, and the smoothness in $\pm e_n$ is achieved because ${\lim_{t\to 1} f'(t)=\infty}$. 
\begin{proposition} \label{prop:O+G-diff-Xn}
The space $(X_n,\ntres{\cdot}_n)$ is G\^ateaux smooth.
\end{proposition}
\begin{proof}
By induction, for $n=0$, we know that the norm $\nnn_0$ on $X_0$ was G\^ateaux smooth by hypothesis.
By equation \eqref{eq:gat-oct-sphere} we know that $\ntres{\cdot}_n$ is the Minkowski functional
of the $1$-level set of the function
\[
g(x):=\ntres{P_{n-1}x}_{n-1}+f_n(|x_n|).
\]
By inductive assumption, $\ntres{\cdot}_{n-1}$ is G\^ateaux smooth (except
the origin). Hence $g(x)$ is also a G\^ateaux smooth convex
function in its domain,
except for the origin and
possibly the point $\pm e_n$. Hence there is a unique
tangent hyperplane to the graph of $g$ at $x$ and this immediately
implies that there is also a unique tangent hyperplane to the Minkowski
functional of $S_{X_n}$. In other words, $\ntres{\cdot}_n$ is
G\^ateaux differentiable at $x$. The remaining case when $x=\pm e_n$
is clear, since the only tangent hyperplane at this point is the
kernel of the $n$-th coefficient functional on $X$, yielding
G\^ateaux smoothness again. 
\end{proof}

\begin{corollary}\label{cor:O+Gdiff-Y}
The space $(Y,\ntres{\cdot})$ is G\^ateaux smooth.
\end{corollary}
\begin{proof}

Given $x,h\in Y$, there exists $n\in NN$ so that both $x,h\in X_n$, so the result follows from the previous one. 
 \end{proof}

In the remaining part of the section, we will
prove that $\ntres{\cdot}$ is G\^ateaux smooth on the whole $X$.
This is the most delicate part of the argument.
Our Banach space $(X,\ntres\cdot)$ is the completion of
the normed space
$(Y,\ntres\cdot)=\cup_n (X_n,\ntres\cdot_n)$. To
prove that the final norm is G\^ateaux smooth, it suffices to
prove that $\ntres\cdot$ has a directional derivative
at any point $0\ne x\in X$ with respect
to a dense set of directions, in our case, for any $h\in Y$. If $x\in Y$ then this follows directly from
our inductive construction, as both $x$ and the direction $h\in Y$
are contained in some $X_n$ and we know that $\ntres\cdot_n$
is G\^ateaux smooth. It remains to deal with the
delicate case $x\in X\setminus Y$.

\medskip

In what follows, we assume without loss of generality that $x=x_0+\sum_{j=1}^\infty x_j e_j$
(where of course  $\sum_n|x_n|<\infty$),
$\ntres{x}=1$. Since $x^n=x_0+\sum_{j=1}^n x_j e_j$,
 $\ntres{x^n}\le1$ for all $n\in\mathbb{N}$.

In order to prove the G\^ateaux smoothness at a fixed point $x$ we will
obtain an estimate of the function (of parameter $t$)
\[
\phi_{x,h}(t): =\ntres{x+th}-\ntres{x}
\]
for every $h\in Y$.
In the end this will lead to the desired conclusion because
\[
\phi'_{x,h}(0)=\frac{\partial \ntres{x}}{\partial h}.
\] 

We start by collecting some simple observations concerning the functions $\phi_{x,h}$. The proof is omitted, as it is immediate.

\begin{lemma}\label{lem:directional-properties}
 For any $\tau>0$,


\begin{align*}
    \phi_{x,\tau h}(t)&=\ntres{x+t\tau h}-\ntres{x}=
\phi_{x,h}(\tau t);\\
\phi_{\tau x,\tau h}(t)&=\ntres{\tau x+t\tau h}-\ntres{\tau x}=
\tau \phi_{x,h}(t);\\
\phi_{\tau x,h}(t)&=\ntres{\tau x+t\tau (\frac1\tau h)}-\ntres{\tau x}=
\tau \phi_{x,\frac1\tau h}(t)=\tau\phi_{x,h}(\frac1\tau t).
\end{align*}
Also, for fixed $x\in X,h\in Y,t\in\mathbb{R}$, we have
\begin{equation*}
\phi_{x,h}(t)=\lim_n\phi_{x^n,h}(t).
\end{equation*}
\end{lemma}

Our strategy is to show that
the sequence of functions $\phi_{x^n,h}$ yields the estimates needed for
G\^ateaux smoothness
at $x$.

\vskip5mm

We will split the argument into two cases through the well-known notion of orthogonality initially introduced by Birkhoff at \cite{Birk35} and studied by James in \cite{Ja47,Ja47-2}.

\begin{definition}
    Given a Banach space $(X,\nn)$, and $x,h\in X\backslash\{0\}$. It is said that $x$ is  \textit{Birkhoff--James orthogonal} to $h$ if for every $t\in \RR$
    \[\|x+th\|\geq \|x\|.\]
\end{definition}

In our context, $x$ is Birkhoff--James orthogonal to $h$ if and only if $\phi_{x,h}(t)\geq 0$ for every $t\in \RR$. Geometrically, the above definition allows us to think of $h$ as a ``tangent direction'' on $x$, meaning that the vector $x+th$ belongs to
a tangent hyperplane of the multiple of the unit
ball $\ndos{x} B_{X}$ at $x$. Notice that if $\nn$ is G\^ateaux at $x$, then $\phi_{x,h}'(0)=0$ for any $h$ tangent direction at $x$. Finally, one of the key properties of the renorming is that the norms are inductively constructed so that the sliced of unit ball $B_{n+1}$ in the direction of $e_{n+1}$ are homothetic copies of the previous step $B_n$---recall Figure \ref{fig:n-ball}. This implies that Birkoff--James orthogonality for a projection is preserved at further steps.

\begin{corollary}\label{cor:preserve-tangent}
    Let $x\in (X,\nnn)$ such that $x^n\neq 0$. If $x^n$ is Birkoff--James orthogonal to a given direction $h\in X_n$, then $x^{n+1}\in X_{n+1}$ is Birkoff--James orthogonal to $h$.
\end{corollary}
\begin{proof}
    By Lemma \ref{lem:nnn-homotecy},
    \[(1-f_{n+1}(\frac{|x_{n+1}|}{\ntres{x^{n+1}}}))\phi_{x^{n+1},h}(t)=\phi_{x^{n},h}(t), \]
    and the result follows.
\end{proof}

 \subsection{Case 1: $x^n$ is Birkoff-James orthogonal to $h\in X_n$} \label{subsec:O+G-tangent}
 
 In this part, we will assume that the direction $h$ belongs to $X_n$ for a given $n\in \NN$, and furthermore, $x^n\in X_n$ is Birkoff--James orthogonal to $h$.
 
 As  $x^n$ is a smooth point, $\phi'_{x^n,h}(0)=0$,
or equivalently
\begin{equation}\label{der-0}
\phi_{x^n,h}(t)=o(t).
\end{equation}

Let us now estimate $\phi_{x^{n+1},h}$.




From the construction of the renormings $\nnn_n$---see Lemma \ref{lem:nnn-homotecy} and Corollary \ref{cor:nnn-homotecy}---we have that for every $\lambda>0$ there is some
$\tilde\lambda\le\lambda$ such that
\[
\lambda S_{X_{n+1}}\cap (X_n+x_{n+1}e_{n+1})=
\tilde\lambda S_{X_n}+x_{n+1}e_{n+1}.
\]

i.e. for a fixed value $x_{n+1}$ we have
\begin{equation*}
\text{$\ntres{x^n}=\tilde\lambda \iff$ $\ntres{x^n+x_{n+1}e_{n+1}}=\ntres{x^{n+1}}=\lambda$.}
\end{equation*}

Take $\tilde P:=\frac{\tilde\lambda}{\ntres{x^n+th}}(x^n+th)$, the point on the ray from the origin to the point $x^n+th$ which has norm $\tilde\lambda=\ntres{x^n}$. 


Now, put $P:=\tilde P+x_{n+1}e_{n+1}$.
Then $P_n(P)=\tilde P$ and $\ntres{P}=\lambda=\ntres{x^{n+1}}$.

Denote $R:=\frac{\lambda}{\ntres{x^{n+1}+th}}(x^{n+1}+th)$
the point of intersection
of the ray from zero to $(x^{n+1}+th)$ with $\lambda S_{X_{n+1}}$.
\medskip
We claim that $\ntres{P_n(R)}\ge\tilde\lambda$ as, from a simple geometric argument, we deduce that $R$ must project farther away from the origin than $\tilde P$---see Figure \ref{fig:claim-cone} below. Indeed, we may write 
\[R=\frac{\ntres{x^{n+1}}}{\ntres{x^{n+1}+th}}((\frac{\ntres{x^n+th}}{\ntres{x^n}}-1)\tilde P+P),\]
so $R\in \sp\{P,\tilde P\}$. Notice that $P,R\in \sp\{ P,\tilde P\} \cap \lambda S_{X_{n+1}}$, but $\ntres{\tilde P}\leq \lambda$. Then, for every $z\in \conv(\tilde P,P)$, it is satisfied that $\ntres{z}\leq \lambda$.

\medskip 
Then, the intersection of $\sp\{x^{n+1}+th\}\cap\conv(\tilde P,P) $ is a unique point $S$, that clearly satisfies $P_n(S)=\tilde P$ and $\ntres{S}\leq \lambda$. As both ${R,S \in \sp\{x^{n+1}+th\}}$, but $\ntres{R}\geq \ntres{S}$, we finally deduce that
\[\ntres{P_n(R)}\geq \ntres{P_n(S)}= \ntres{\tilde P}.\]
So, the claim is proved.

\begin{figure}[htb!]
\begin{center}
\includegraphics[width=4in]{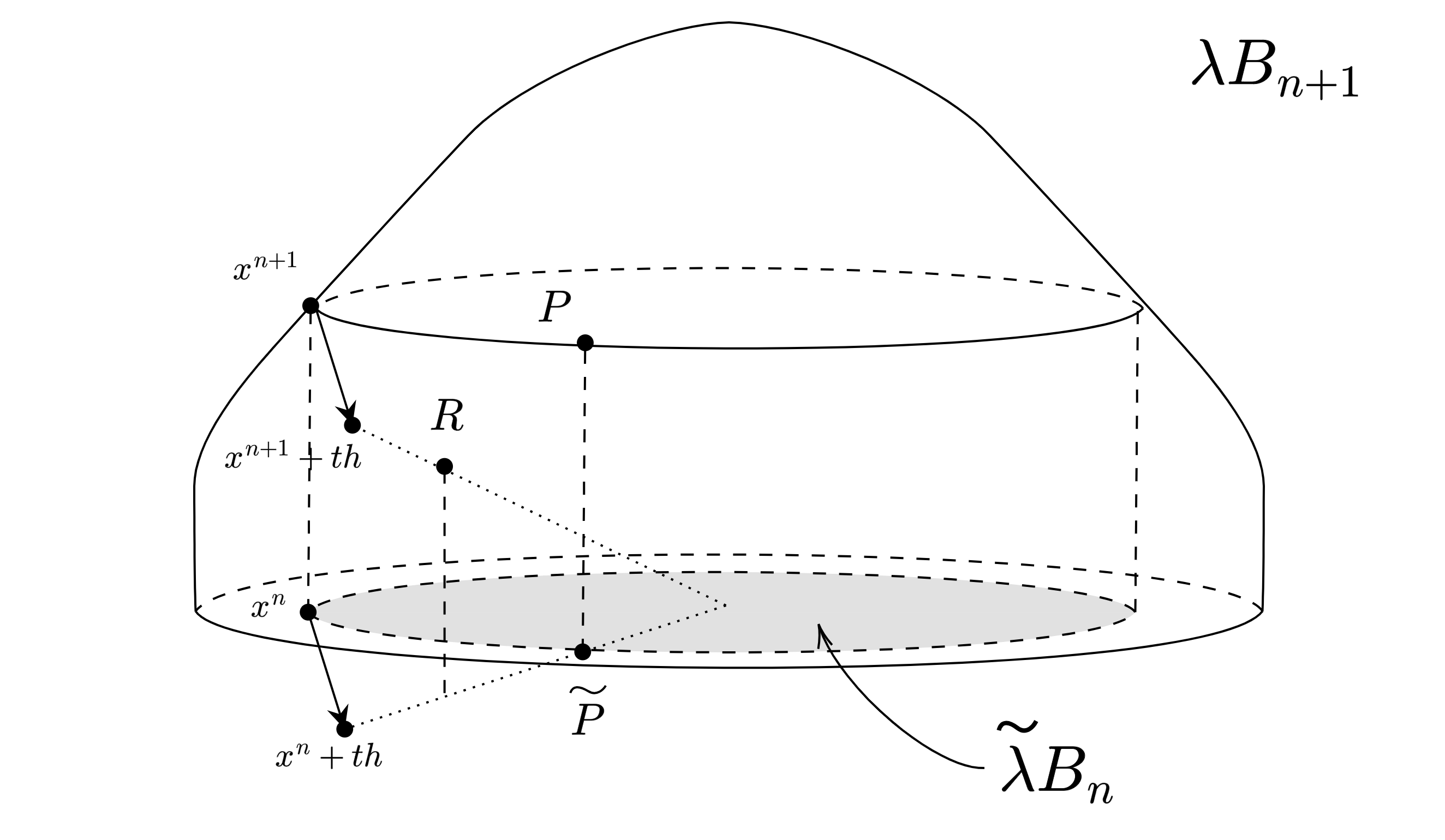}
\caption{Geometric interpretation of the claim.}
\label{fig:claim-cone}
\end{center}
\end{figure}

But then, by the claim

\[\dfrac{\ntres{x^{n+1}}}{\ntres{x^{n+1}+th}}\dfrac{\ntres{x^{n}+th}}{\ntres{x^{n}}} \ntres{\tilde P} = \ntres{P_n(R)}\geq \ntres{\tilde P},  \]

So we just deduced

\begin{equation*}
   \dfrac{\ntres{x^{n+1}}}{\ntres{x^{n}}}\geq 
 \dfrac{\ntres{x^{n+1}+th}}{\ntres{x^{n}+th}}.
\end{equation*}

Using Corollary \ref{cor:preserve-tangent} and this last inequality,
\begin{align*}
   0\leq  \phi_{x^{n+1}, h}(t)&= \ntres{x^{n+1}+th}- \ntres{x^{n+1}}\\
    &\leq \dfrac{\ntres{x^{n+1}}}{\ntres{x^{n}}} \ntres{x^{n}+th} - \ntres{x^{n+1}}\\
    &=\dfrac{\ntres{x^{n+1}}}{\ntres{x^{n}}} \phi_{x^{n}, h}(t).
\end{align*}


Now, we may fix $n$ large enough so that  $h\in X_n$ and
$\ntres{x^n}\ge\frac12\ntres{x}$.
Proceeding inductively as above, we get the estimate for every $m\ge n$:

\begin{equation}\label{teles}
\phi_{x^{m}, h}(t)\le\frac{\ntres{x^{m}}}{\ntres{x^{n}}}\phi_{x^n, h}(t)\le
2\phi_{x^n,h}(t).
\end{equation}

Then, passing to a uniform limit for $m\to\infty$ this clearly implies that

\[
\phi_{x, h}(t)\le2\phi_{x^n,h}(t),
\]
which using \eqref{der-0} means that $\phi_{x,h}(t)=o(t)$ and so 
$\frac{\partial\ntres{x}}{\partial h}=0$.

\subsection{Case 2: no projection $x^m$ is Birkoff--James orthogonal to the direction $h$}\label{subsec:O+G-not-tangent}

Let us pass to the general case. We may assume  $\ntres{x}=\ntres{h}=1$. For sufficiently large $N\in\mathbb{N}$,
 $h\in X_N$, 
and $\ntres{x-x^N}<\frac1{16}$. In what follows, 
 we tacitly assume that $m\geq n>N$, i.e.
without loss of generality, we assume that always $\ntres{x^n}>\frac{15}{16}$ and
$\ntres{x^n-x^m}<\frac18$. We may also consider $|t|<\frac{15}{2\cdot 16}$.

For $m$, let $h=h_m+C_m x^m$ be a (unique) decomposition such that $h_m$ is a tangent direction at $x^m$---see Figure \ref{fig:tg-decomp-m}.

\begin{figure}[htb!]
\begin{center}
\includegraphics[width=4in]{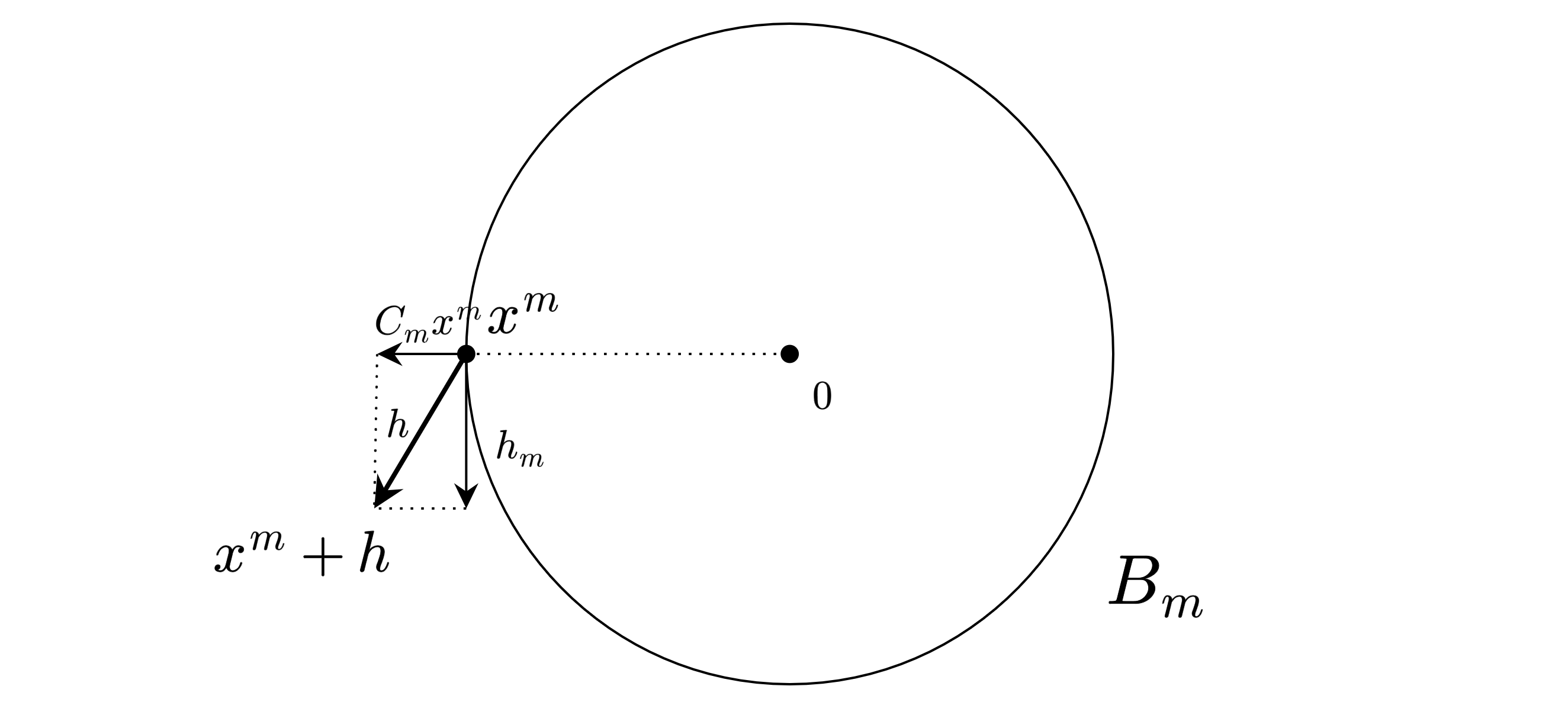}
\caption{Decomposition of a direction $h\in X_m$}
\label{fig:tg-decomp-m}
\end{center}
\end{figure}

Notice that, considering  $g_m\in S_{X_m^*}$, the (unique) norming functional
for  $x^m\in X_m$ (i.e., $\la g_m,x^m \ra=\ntres{x^m}$), it is well known that this functional $g_m$ is the one describing the value of the directional derivatives at $x^m$, i.e., for any direction $y\in X_m\backslash \{0\}$
\[
\frac{\partial\ntres{x^m}}{\partial y}=\la g_m,y \ra.
\]

In particular, we have that $\la g_m, h_m\ra =0$. Geometrically, the unique tangent hyperplane of $\ntres{x^m}B_{X_m}$ at $x^m$ is exactly $x^m+\ker(g_m)$.

As $|\la g_m,h\ra|\le1$ and $\la g_m,x^m\ra= \ntres{x^m}>\frac{15}{16}$, we have

\[
\frac{\partial\ntres{x^m}}{\partial h} =\la g_m,h\ra =C_m \la g_m, x^m \ra =C_m\ntres{x^m},
\]
 from where we deduce that
\begin{equation}\label{c-n est}
|C_m|\le\frac{16}{15}.
\end{equation}

Recall that, by the construction of the norm $\nnn$, tangent directions at one dimension are preserved in further dimensions (see Corollary \ref{cor:nnn-homotecy}), so $h_n\in \ker g_m$, i.e. $\la g_m, h_n\ra =0$, whenever $m\ge n$. Thus, as
\[
g_m(h)=g_m(h_m+C_m x^m)=g_m(h_n+C_n x^n),
\]
we deduce that $C_m g_m(x^m)=C_n g_m(x^n)$ and then we get
\begin{align*}
    0&=C_m g_m(x^m)-C_n g_m(x^n)\\
    &=(C_m-C_n) g_m(x^n)+C_m\sum_{i=n+1}^m x_i g_m(e_i).
\end{align*}

Now, if we combine this equation above with the lower bound
\[
|g_m(x^n)|\geq |g_m(x^m)|-\ntres{x^m-x^n}>\frac{15}{16}-\frac18,
\]

 and using also the upper bound $|C_m|\leq\frac{16}{15}$ from \eqref{c-n est}, we achieve

\[
|C_m-C_n|\le\frac{|C_m|}{|g_m(x^n)|}\sum_{i=n+1}^m|x_i|<2\sum_{i=n+1}^m|x_i|
\]
and, in particular,

\begin{equation}\label{c-m-n}
|C_{n+1}-C_n|<2|x_{n+1}|.
\end{equation}

 With this, we will get the estimation
\begin{equation}\label{eqq-h}
      \ntres{h_{n+1}-h_n} <4|x_{n+1}|.
\end{equation}
 Indeed, using equations \eqref{c-n est} and \eqref{c-m-n}, 
\begin{align*}
    \ntres{h_{n+1}-h_n}&=\ntres{C_{n+1}x^{n+1}-C_n x^n} \\
    &\le \ntres{(C_{n+1}-C_n)x^{n+1}+C_n x^{n+1}-C_n x^n}\\
    &\le\ntres{(C_{n+1}-C_n)x^{n+1}}+
\ntres{C_n x^{n+1}-C_n x^n}\\
&\le 2|x_{n+1}|+ C_n\ntres{x_{n+1}e_{n+1}}\\&<4|x_{n+1}|.
\end{align*}

Now, using Lemma \ref{lem:directional-properties}, 

\begin{align*}\label{h-hn}
    \phi_{x^n,h}(t)&=\ntres{x^n+th_n+tC_n x^n}-\ntres{x^n}\\
    &=\ntres{(1+tC_n)x^n+th_n}-\ntres{(1+tC_n)x^n}+\ntres{(1+tC_n)x^n}-\ntres{x^n}\\
    &=\phi_{(1+tC_n)x^n, h_n}(t)+tC_n\ntres{x^n}\\
    &=(1+tC_n)\phi_{x^n, h_n}(\frac1{1+tC_n}t)+tC_n\ntres{x^n}.
\end{align*}



Notice that the above formula gives an expression of $\phi_{x^n,h}(t)$ that depends on the tangent direction $h^n$. By (re-)writing the formula for the $n+1$-dimension, we would get an expression depending on the next tangent direction $h_{n+1}$. The idea is that, thanks to \eqref{eqq-h}, we can estimate the function $\phi_{x^{n+1},h}$, but still use the previous tangent direction $h_n$. In fact, by writing the formula and adding and subtracting the vector ${v_n:=\frac1{1+tC_{n+1}}th_{n}} $, we get


\begin{align*}
    \phi_{x^{n+1},h}(t)&=(1+tC_{n+1})\phi_{x^{n+1}, h_{n+1}}
(\frac1{1+tC_{n+1}}t)+tC_{n+1}\ntres{x^{n+1}}\\
&=(1+tC_{n+1})(\ntres{x^{n+1}+\frac1{1+tC_{n+1}}th_{n+1}}-\ntres{x^{n+1}})\\
&\ +tC_{n+1}\ntres{x^{n+1}}\\
&=(1+tC_{n+1})(\ntres{x^{n+1}+v_n +\frac1{1+tC_{n+1}}th_{n+1} -v_n}-\ntres{x^{n+1}})\\
&\ +tC_{n+1}\ntres{x^{n+1}}.\\
\end{align*}

We may now use the triangular inequality on the last step of the formula above (where the vectors $v_n$ are introduced). This means that, for a certain error $E_n(t)$, we are able to express
\begin{equation}\label{neci}
    \phi_{x^{n+1},h}(t)  
=\phi_{(1+tC_{n+1})x^{n+1}, h_n}(t)+tC_{n+1}\ntres{x^{n+1}}+E_{n}(t)
\end{equation}

where the error is estimated by
\begin{equation}\label{err-e}
|E_n(t)|\le |t|\ntres{h_{n+1}-h_n}<4|t||x_{n+1}|.
\end{equation}

Notice that comparing the two formulas that we achieved for $\phi_{x^{n+1},h}$, we have that the error is exactly
\begin{equation}\label{eq:err-tang}
   E_n(t)= \phi_{(1+tC_{n+1})x^{n+1}, h_{n+1}}(t) -\phi_{(1+tC_{n+1})x^{n+1}, h_n}(t).
\end{equation}

We can then use \eqref{neci}, getting
\begin{align}
\phi_{x^{n+k},h}(t)-\phi_{x^{n},h}(t)&=
\sum_{j=0}^{k-1}(\phi_{x^{n+j+1},h}(t)-\phi_{x^{n+j},h}(t))\nonumber \\
&=A+B+C\label{eq:ABC},
\end{align}

where we have the three terms
\begin{align*}
    A&:= \sum_{j=0}^{k-1}(\phi_{(1+tC_{n+j+1})x^{n+j+1}, h_{n+j}}(t)-
\phi_{(1+tC_{n+j})x^{n+j}, h_{n+j}}(t)),\\
B&:= \sum_{j=0}^{k-1} t(C_{n+j+1}-C_{n+j})\ntres{x^{n+j+1}}
+\sum_{j=0}^{k-1}tC_{n+j}(\ntres{x^{n+j+1}}-\ntres{x^{n+j}}),\\
C&:=E_{n+k}(t)-E_n(t).
\end{align*}

We can estimate separately the three of them. First, notice that using previous estimations on 
$\nnn$ and using inequalities \eqref{c-n est} and \eqref{c-m-n} to control $|C_{n+j+1}|$ and  $|C_{n+j+1}-C_{n+j}|$ respectively, we get the bound

\[|B|\leq 2|t|\sum_{j=0}^{k-1} |x_{n+j+1}|
+\frac{16}{15}|t|\sum_{j=0}^{k-1}|x_{n+j+1}|.\]

Also, the bound on the error that we had at inequality \eqref{err-e} yields
\[|C|\leq 4|t|| x_{n+k}|+4|t|| x_n|. \]

So, with respect to $B$ and $C$---and for the sake of simplicity---we might take the same bound,
\begin{align}
    |B|\leq 4|t|\sum_{j=0}^{\infty}|x_{n+j}|\label{eq:bound-B}\\
     |C|\leq 4|t|\sum_{j=0}^{\infty}|x_{n+j}|\label{eq:bound-c}
\end{align}
Now, we only need to estimate $A$. Notice that we may split again
\begin{align*}
    |A| \leq |A_1| + |A2|, 
\end{align*}

where 

\begin{align}
    A_1&:= \phi_{(1+tC_{n+k})x^{n+k}, h_{n+k-1}}(t)-
\phi_{(1+tC_{n})x^{n}, h_{n}}(t) \label{eq:defA1}\\
A_2&:= \sum_{j=0}^{k-2}|\phi_{(1+tC_{n+j+1})x^{n+j+1}, h_{n+j}}(t)-
\phi_{(1+tC_{n+j+1})x^{n+j+1}, h_{n+j+1}}(t)| \label{eq:defA2}
\end{align}





For $A_2$, by the error expression at \eqref{eq:err-tang}, we have for $j=0,\dots, k-1$
\begin{align}
    |\phi_{(1+tC_{n+j+1})x^{n+j+1}, h_{n+j}}(t)-
\phi_{(1+tC_{n+j+1})x^{n+j+1}, h_{n+j+1}}(t)| &\leq |E_{n+j}(t)| \nonumber\\
&\leq 4|t||x_{n+j+1}|. \label{ppp33}
\end{align}

So, $A_2$ can be bounded as 
\begin{equation}\label{eq:est-A-2}
   A_2 \leq 4|t|\sum_{j=0}^{k-2}|x_{n+j+1}|= 4|t|\sum_{j=0}^{k-1}|x_{n+j}|.
\end{equation}

For $A_1$, we may express the first of its two terms as


\begin{align} 
  \phi_{(1+tC_{n+k})x^{n+k}, h_{n+k-1}}(t)  =& \Big(\phi_{(1+tC_{n+k})x^{n+k},h_{n+k-1}}(t) - \phi_{(1+tC_{n+k})x^{n+k},h_{n}}(t)\Big) \label{fgh} \\
  &+ \phi_{(1+tC_{n+k})x^{n+k},h_{n}}(t)\nonumber
\end{align}

For the first summand of \eqref{fgh}

\begin{align}
       |\phi_{(1+tC_{n+k})x^{n+k},h_{n+k-1}}(t) - \phi_{(1+tC_{n+k})x^{n+k},h_{n}}(t)|   &\leq |t|\ntres{h_{n+k-1}-h_n} \nonumber\\
          &\leq 4|t|\sum_{j=1}^{k-1}|x_{n+j}|\label{fghki}.
\end{align}


And for the second term of \eqref{fgh}, as $x^n$ is Birkoff--James orthogonal to $h_n$, we can reduce the last term of the above formula to the previous Case 1 in Subsection \ref{subsec:O+G-tangent}, so applying \eqref{teles} and re-writing through Lemma \ref{lem:directional-properties}, we get
\begin{align*}
    |\phi_{(1+tC_{n+k})x^{n+k},h_{n}}(t)|\le
|(1+tC_{n+k})\frac{\ntres{x^{n+k}}}{\ntres{x^{n}}}
\phi_{x^{n}, h_{n}}(\frac1{1+tC_{n+k}}t)|,
\end{align*}

and recalling the estimates of $|C_m|\leq \frac{16}{15}$ in \eqref{c-n est}, and the initial assumptions on the norms $\ntres{x^m}$ and that we took $|t|<\frac{15}{2 \cdot 16}$, we get that for any $k\in \NN$

\begin{align}
 |\phi_{(1+tC_{n+k})x^{n+k},h_{n}}(t)| &\leq \frac{16}{15}(1+ \frac{15}{2 \cdot 16}\frac{16}{15})|\phi_{x^{n}, h_{n}}(\frac1{1+tC_{n+k}}t)|\nonumber \\
 &\leq 2 \phi_{x^{n}, h_{n}}(\frac1{1+tC_{n+k}}t)\label{eq:boundA12},
 \end{align}

where for every $k\in \NN$, $\frac{1}{2}\leq 1+tC_{n+k} \leq \frac{3}{2}$ and 
\begin{equation}\label{eq:indep-k}
    \dfrac{2|t|}{3} \leq \left| \dfrac{t}{1+tC_{n+k}}\right| \leq 2|t|.
\end{equation}
 
So, by \eqref{fgh} and \eqref{fghki}
we get the estimation on $A_1$,
\begin{align}
    |A_1| &=  |\phi_{(1+tC_{n+k})x^{n+k}, h_{n+k-1}}(t)-
\phi_{(1+tC_{n})x^{n}, h_{n}}(t) | \nonumber \\
    &\leq |\phi_{(1+tC_{n+k})x^{n+k}, h_{n+k-1}}(t)|+
|\phi_{(1+tC_{n})x^{n}, h_{n}}(t)|\nonumber \\
&\leq 4|t|\sum_{j=1}^{k-1}|x_{n+j}|+  |\phi_{(1+tC_{n+k})x^{n+k}, h_{n}}(t)|  + |\phi_{(1+tC_{n})x^{n}, h_{n}}(t)|.\label{eq:est-A-1}
\end{align}

Combining the bounds of $A_1$ and $A_2$ (adding equations \eqref{eq:est-A-1} and \eqref{eq:est-A-2} respectively), we get the remaining estimation of $A$,
\begin{equation}\label{eq:est-A}
    |A|\leq    8|t|\sum_{j=1}^{k-1}|x_{n+j}|+   |\phi_{(1+tC_{n+k})x^{n+k}, h_{n}}(t)|  + |\phi_{(1+tC_{n})x^{n}, h_{n}}(t)|
\end{equation}

Combining together the bounds of $A$, $B$ and $C$ (equations \eqref{eq:est-A},  \eqref{eq:bound-B} and \eqref{eq:bound-c} respectively) and returning to \eqref{eq:ABC}, we finally achieve
\begin{align}
    |\phi_{x^{n+k},h}(t)-\phi_{x^{n},h}(t)| &\le A+B+C \nonumber\\
  &\leq   (8+4+4) |t|\sum_{j=0}^{\infty}|x_{n+j}| \nonumber \\ 
  &\ +  |\phi_{(1+tC_{n+k})x^{n+k}, h_{n}}(t)|  + |\phi_{(1+tC_{n})x^{n}, h_{n}}(t)|
\end{align}
\label{konec}

And finally applying \eqref{eq:boundA12} to bound the last two terms in the equation above, we reach
\begin{align*}
        |\phi_{x^{n+k},h}(t)-\phi_{x^{n},h}(t)| & \leq  16|t|\sum_{j=0}^{\infty}|x_{n+j}| +2 \phi_{x^{n}, h_{n}}(\frac1{1+tC_{n+k}}t) \\
        &\ +2 \phi_{x^{n}, h_{n}}(\frac1{1+tC_{n}}t)
\end{align*}
And this last inequality implies that $\phi_{x,h}(t)$ is differentiable
at $t=0$---with derivative equal to $\lim_{n\to\infty}C_n$. 
Indeed, given any $\varepsilon>0$ we may pick $n$ large enough
 so that
$\sum_{j=0}^{\infty}|x_{n+j}|<\varepsilon$. We know that
$\phi_{x^n,h}(t)$ and $\phi_{x^n,h_n}(t)$ are differentiable at 
$t=0$ by the previous case---with values being equal to $C_n$ and $0$,
respectively. So,

\begin{equation}\label{konec-2}
|\phi_{x^{n+k},h}(t)-\phi_{x^{n},h}(t)|\le
16\varepsilon |t|+o(t)
\end{equation}
where the $o(t)$ estimate is independent of $k$ because of \eqref{eq:indep-k}. This finishes the argument and the proof of the main result.

\section*{Acknowledgements}
The present work was done while the first author was on a PhD stay at the Department of Mathematics of the Faculty of Electrical Engineering at the Czech Technical University in Prague, and he is grateful for all the support and kindness received there. The first author was supported by  Generalitat Valenciana (through Project PROMETEU/2021/070 and the predoctoral fellowship CIACIF/2021/378) and by MICIU/AEI/
10.13039/501100011033 and by ERDF/EU (through Projects PID2019-105011GB-I00, PID2021-122126NB-C33 and PID2022-139449NB-I00). The second author was supported by GA23-04776S and SGS24/052/OHK3/1T/13 of CTU in Prague.

\end{document}